\documentclass[11]{amsart}

\usepackage{amsfonts}
\usepackage{amscd}
\usepackage{amsmath, mathrsfs, amssymb}
\usepackage{amsthm}
\usepackage{setspace}
\usepackage{hyperref}
\usepackage{color}
\usepackage{epsfig}
\usepackage{here}
\usepackage{graphicx}
\usepackage[all]{xy}
\usepackage{psfrag}
\usepackage{graphicx,transparent}
\usepackage{enumerate}
\usepackage{caption}

\theoremstyle{plain}
\newtheorem{theorem}{Theorem}[section]
\newtheorem{lemma}[theorem]{Lemma}
\newtheorem{proposition}[theorem]{Proposition}
\newtheorem{corollary}[theorem]{Corollary}

\theoremstyle{definition}
\newtheorem{remark}[theorem]{Remark}

\newcommand{\MM}{\mathcal M}

\newcommand{\BM}{\overline{\mathcal M}}

\newcommand{\PP}{\mathcal P}

\newcommand{\calX}{\mathcal X}

\newcommand{\RR}{\mathbb R}

\newcommand{\HH}{\mathcal H}

\newcommand{\calL}{\mathcal L}

\newcommand{\hyp}{\operatorname{hyp}}
\newcommand{\GL}{\operatorname{GL}}

\newcommand{\even}{\operatorname{even}}
\newcommand{\odd}{\operatorname{odd}}

\newcommand{\wt}{\widetilde}

\newcommand{\bbC}{\mathbb C}

\newcommand{\bbP}{\mathbb P}

\newcommand{\bbR}{\mathbb R}
\newcommand{\bbZ}{\mathbb Z}

\newcommand{\SL}{\operatorname{SL}}

\newcommand{\diag}{\operatorname{diag}}

\begin{document}

\makeatletter
	\@namedef{subjclassname@2010}{%
	\textup{2010} Mathematics Subject Classification}
	\makeatother

\title{Affine geometry of strata of differentials}

\author{Dawei Chen}
\address{Department of Mathematics, Boston College, Chestnut Hill, MA 02467}
\email{dawei.chen@bc.edu}


\date{}

\thanks{The author is partially supported by NSF CAREER Award DMS-1350396.}

\begin{abstract}
Affine varieties among all algebraic varieties have simple structures. For example, an affine variety does not contain any complete algebraic curve. In this paper we study affine related properties of strata of $k$-differentials on smooth curves which parameterize sections of the $k$-th power of the canonical line bundle with prescribed orders of zeros and poles. We show that if there is a prescribed pole of order at least $k$, then the corresponding stratum does not contain any complete curve. Moreover, we explore the amusing question whether affine invariant manifolds arising from Teichm\"uller dynamics are affine varieties, and confirm the answer for Teichm\"uller curves, Hurwitz spaces of torus coverings, hyperelliptic strata as well as some low genus strata. 
\end{abstract}

\maketitle


\section{introduction}
\label{sec:intro}

For a positive integer $k$, let $\mu = (m_1, \ldots, m_n)$ be an integral partition of $k(2g-2)$, i.e. $m_i \in \bbZ$ for all $i$ and $\sum_{i=1}^n m_i = k(2g-2)$. The stratum of $k$-differentials $\HH^k(\mu)$ parameterizes pairs $(C, \xi)$ where $C$ is a smooth complex curve of genus $g$ and $\xi$ is a (possibly meromorphic) section of $K_C^{\otimes k}$ such that $(\xi)_0 - (\xi)_\infty = \sum_{i=1}^n m_i p_i$ for distinct points $p_1, \ldots, p_n \in C$. If we consider differentials up to scale, then the corresponding stratum of $k$-canonical divisors $\PP^k(\mu)$ parameterizes the underlying divisors $ \sum_{i=1}^n m_i p_i$, and $\HH^k(\mu)$ is a $\bbC^*$-bundle over $\PP^k(\mu)$. We also write $\HH(\mu)$ and $\PP(\mu)$ for the case $k=1$. 

Abelian and quadratic differentials, i.e. the cases $k=1$ and $k=2$ respectively, have broad connections to flat geometry, billiard dynamics, and Teichm\"uller theory. Recently their algebraic properties (also for general $k$-differentials) have been investigated intensively, which has produced fascinating results in the study of both Teichm\"uller dynamics and moduli of curves. We refer to \cite{ZorichFlat, WrightLectures, Bootcamp, EskinMirzakhani, EskinMirzakhaniMohammadi, Filip, FarkasPandharipande, BCGGM1, BCGGM2} for related topics as well as some recent advances. 

Let $\MM_g$ be the moduli space of smooth genus $g$ curves (and $\MM_{g,n}$ the moduli space with $n$ ordered marked points). One of the central questions in the study of moduli of curves is to determine complete subvarieties in $\MM_g$. Diaz \cite{Diaz} first proved that the dimension of any complete subvariety in $\MM_g$ is at most $g-2$ (see \cite{Looijenga, GrushevskyKrichever} for alternate proofs). Nevertheless, it is unknown whether this bound is sharp. In particular, we do not know whether $\MM_4$ contains a complete algebraic surface. 

Despite that little is known about the global geometry of the strata of differentials (see \cite{LooijengaMondello, Gendron, Barros} for some special cases), one can similarly ask the question about complete subvarieties in the strata. Our first result is as follows. 

\begin{theorem}
\label{thm:complete}
If $\mu$ has an entry $m_i \leq -k$, then there is no complete curve in $\HH^k(\mu)$ and $\PP^k(\mu)$.  
\end{theorem}

The case $k=1$ yields the following result. 

\begin{corollary}
Every stratum of strictly meromorphic abelian differentials or canonical divisors does not contain a complete curve. 
\end{corollary}

Next we turn to the case of holomorphic abelian differentials, i.e. when $k=1$ and all entries of $\mu$ are positive. A novel idea in the study of $\MM_g$ is to construct special stratifications of $\MM_g$ (see e.g. \cite{FontanariLooijenga}). Denote by $\HH$ the Hodge bundle parameterizing holomorphic abelian differentials, which is a rank-$g$ vector bundle over $\MM_g$, and let $\PP$ be the projectivization of $\HH$ parameterizing holomorphic canonical divisors. Then the strata $\PP(\mu)$ with holomorphic signatures $\mu$ form a stratification of $\PP$. For special $\mu$, $\PP(\mu)$ can have up to three connected components, due to hyperelliptic and spin structures (\cite{KontsevichZorich}),  
and we label by $\hyp$, $\even$ and $\odd$ to distinguish these strata (components). It was asked in \cite[Problem 3.7]{Mondello} whether the strata $\PP(\mu)$ are affine. If the answer is yes in general, then it would follow that there is no complete curve in $\PP(\mu)$. 

Knowing the affinity of the strata has another implication. The group $\GL^+_2(\bbR)$ acts on the strata $\HH(\mu)$ of holomorphic abelian differentials by varying the shape of the corresponding flat surfaces. By \cite{EskinMirzakhani, EskinMirzakhaniMohammadi} (the normalizations of) all $\GL^+_2(\bbR)$-orbit closures (under the analytic topology of $\HH(\mu)$) are locally cut out by homogeneous real linear equations of period coordinates, hence they are also called affine invariant manifolds. Here the term ``affine'' is different from what it means in algebraic geometry, as it refers to the locally linear structure. Nevertheless, one can ask the following question: Are affine invariant manifolds affine varieties? More precisely, we project a $\GL^+_2(\bbR)$-orbit closure to $\PP(\mu)$ by modulo scaling of the differentials. Note that any closed subset of an affine variety is still affine (see e.g. \cite[p. 106, Theorem 3]{Mumford}), and by \cite{Filip} all $\GL^+_2(\bbR)$-orbit closures are closed subvarieties in the strata. Therefore, if $\PP(\mu)$ is affine as suspected above, it would imply that all orbit closures in $\PP(\mu)$ are affine. 

Our next result provides evidence for two typical types of orbit closures. The first type is called Teichm\"uller curves, which correspond to closed $\GL^+_2(\bbR)$-orbits in the strata. The other type arises from Hurwitz spaces of torus coverings with $n$ branch points, and the case $n=1$ yields special Teichm\"uller curves. 

\begin{theorem}
\label{thm:hurwitz}
Teichm\"uller curves and Hurwitz spaces of torus coverings are affine. 
\end{theorem}

We also show that the hyperelliptic strata and some low genus strata of holomorphic canonical divisors are affine, see Section~\ref{sec:examples}. 

\subsection*{Acknowledgements} This work was carried out in part when the author visited Seoul National University in December 2016. The author thanks David Hyeon for his invitation and hospitality. The author also thanks Qile Chen, Claudio Fontanari and Sam Grushevsky for helpful communications on related topics. 

\section{Complete curves in the strata}
\label{sec:complete}

In this section we prove Theorem~\ref{thm:complete}. The upshot is an intersection calculation, which was also used to derive a relation between area Siegel-Veech constants and sums of Lyapunov exponents for affine invariant manifolds (see \cite{ChenMoellerAbelian, ChenMoellerQuadratic, EKZ}). The reason we need a pole of order at least $k$ will become clear in the proof. Recall first the following well-known result. 

\begin{lemma}
\label{lem:self}
Suppose $\pi: \calX\to C$ is a complete one-dimensional family of smooth genus $g$ curves for $g\geq 2$, and $S$ is a section. Then the self-intersection $S^2 \leq 0$. Moreover, $S^2 = 0$ iff the family has constant moduli and $S$ is a constant section.  
\end{lemma}

\begin{proof}
By e.g. \cite[p. 309]{HarrisMorrison} we know that $S^2 \leq 0$, and if the family is nonconstant then $S^2 < 0$. Now suppose all fibers of $\pi$ are  isomorphic to a smooth curve $X$. Let $f: \calX \to X$ be the natural projection. If $S$ is not a constant section, then $f|_{S}$ is onto $X$, and let $d$ be the degree of this map. It follows that 
$$- S^2 = S \cdot f^* K_X = f_{*} S \cdot K_X = d (2g-2) > 0, $$
hence $S^2 < 0$. 
\end{proof}

\begin{proof}[Proof of Theorem~\ref{thm:complete}]
Since $\HH^k(\mu)$ is a $\bbC^*$-bundle over $\PP^k(\mu)$, a complete curve in $\HH^k(\mu)$ maps to a complete curve in $\PP^k(\mu)$, hence it suffices to prove the claim for $\PP^k(\mu)$. We can also assume that the zeros and poles as marked points are ordered, as it differs from the unordered setting only by a finite morphism, which does not affect the (non)existence of a complete curve in the respective moduli spaces. Suppose $C$ is a complete curve in $\PP^k(\mu)$. Let $\pi: \calX \to C$ be the universal curve and $S_1, \ldots, S_n$ the sections corresponding to the prescribed zeros and poles of order $m_1, \ldots, m_n$, respectively. Denote by $\omega$ the divisor class of the relative dualizing line bundle of $\pi$. 

Consider first the case $g\geq 2$. Since $k\omega$ and $\sum_{i=1}^n m_i S_i$ restricted to every fiber of $\pi$ both correspond to the $k$-th power of the canonical line bundle of the fiber, there exists a line bundle class $\calL$ on $C$ such that 
\begin{eqnarray}
\label{eq:pullback}
 k \omega = \pi^{*}\calL + \sum_{i=1}^n m_i S_i. 
\end{eqnarray}
Observe the following relations of intersection numbers: 
$$S_i \cdot \omega = - S_i^2, $$
$$ \pi^*\calL \cdot S_i = \deg \calL, $$
$$ S_i\cdot S_j = 0, \quad i\neq j. $$
Intersecting \eqref{eq:pullback} with $S_i$ and using the above relations, we obtain that 
\begin{eqnarray}
\label{eq:section}
 (m_i + k) S_i^2 = - \deg \calL 
 \end{eqnarray}
for all $i$. Next take the square of \eqref{eq:pullback}. Using the above relations along with $(\pi^{*}\calL)^2 = 0$, we obtain that 
\begin{eqnarray}
\label{eq:square}
k^2 \omega^2 = \sum_{i=1}^n (2m_i \deg \calL + m_i^2 S_i^2). 
 \end{eqnarray}

Suppose there exists some $m_j \leq -k$. We discuss separately the cases $m_j = -k$ and $m_j < -k$. If $m_j = -k$, then by \eqref{eq:section} we have $\deg \calL = 0$, and \eqref{eq:square} reduces to $k^2\omega^2 = \sum_{i=1}^n m_i^2 S_i^2$. If
$C$ is not a constant family, by Lemma~\ref{lem:self} we know that $S_i^2 < 0$ for all $i$. On the other hand, $\omega^2$ equals the first $\kappa$-class (also equal to $12\lambda$), which is ample on $\MM_g$ and hence has positive degree on $C$, leading to a contradiction. If the family is constant, then there exists some nonconstant section $S_h$, for otherwise $C$ would parameterize a single point in $\PP^k(\mu)$, hence $S_h^2 < 0$ by Lemma~\ref{lem:self}. Moreover by \eqref{eq:section} and $\deg \calL = 0$, we have $m_h = -k$. Since $S_j^2 \leq 0$ for all $j$, we obtain that $\omega^2 \leq S_h^2 < 0$, contradicting that $\omega^2 = 0$ on a constant family. 

Now suppose $m_i \neq -k$ for all $i$ and there exists some $m_j < -k$. 
Since $\sum_{i=1}^n m_i = k(2g-2) > 0$, there exists some $m_h > 0$. 
Since $S_i^2 \leq 0$ for all $i$, $m_j + k < 0$ and $m_h + k > 0$, it follows from \eqref{eq:section} that $\deg \calL = 0$ and $S_i^2 = 0$ for all $i$. By 
Lemma~\ref{lem:self}, $C$ parameterizes a constant family and all sections $S_i$ are constant, leading to a contradiction. 

For $g=1$, $\PP^k(\mu)$ parameterizes pointed genus one curves $(E, p_1, \ldots, p_n)$ such that $\sum_{i=1}^n m_i p_i$ is the trivial divisor class. Since the fibers of the forgetful map $\MM_{1,n} \to \MM_{1,n-1}$ are not complete for all $n$ and $\MM_{1,1}$ is not complete, a complete curve $C$ in $\PP^k(\mu)$ must parameterize constant $j$-invariant, i.e. $C$ is a complete one-dimensional family of $n$ distinct points $p_1, \ldots, p_n$ on a fixed genus one curve $E$ such that $\sum_{i=1}^n m_i p_i$ is trivial. Fix $p_1$ at the origin so that $(E, p_1)$ becomes an elliptic curve. Then there exists some other $p_i$, say $p_2$, such that $p_2$ varies in $E$ with respect to $p_1$. Since the family is complete, the locus $p_2$ traces out in $(E, p_1)$ is a complete one-dimensional subcurve, which must be $E$ itself. It follows that $p_2$ will meet $p_1$ at some point, contradicting that they are distinct over the entire family. 

Finally for $g=0$, note that $\PP^k(\mu)$ is isomorphic to $\MM_{0,n}$ which is affine, hence it does not contain any complete curve. 
\end{proof}

If $\omega$ is a $k$-differential on $C$, then there exists a canonical cyclic $k$-cover $\pi: \wt{C}\to C$ and an abelian differential $\wt{\omega}$ on $\wt{C}$ such that 
$\pi^{*}\omega = \wt{\omega}^k$ (see e.g.~\cite{BCGGM2}), where $\wt{\omega}$ is determined up to the choice of a $k$-th root of unity. A singularity of $\omega$ of order $m$ gives rise to singularities of $\wt{\omega}$ of order $(m+k)/\gcd (m, k) - 1$. Therefore, if $\HH^k(\mu)$ has all entries $> -k$, then via this covering construction it lifts (as an unramified $k$-cover) into a stratum $\HH(\wt{\mu})$ of holomorphic abelian differentials (here we allow some entries of $\wt{\mu}$ to be zero). In particular, a complete curve in $\HH^k(\mu)$ lifts to a union of complete curves in $\HH(\wt{\mu})$. Combining with Theorem~\ref{thm:complete}, we thus conclude the following result. 

\begin{proposition}
If the strata of holomorphic abelian differentials (or canonical divisors) do not contain any complete curve, then the strata of $k$-differentials (or $k$-canonical divisors) do not contain any complete curve for all $k$. 
\end{proposition}

The study of complete subvarieties in the strata of holomorphic abelian differentials is more challenging and requires some new idea. Even for the minimal (non-hyperellitpic) strata $\PP(2g-2)$, three decades ago Harris \cite[p. 413]{Harris} asked about the existence of complete families of subcanonical points, and it remains open as far as the author knows. We plan to treat this question in future work.   

\section{Affine invariant manifolds}

In this section we prove Theorem~\ref{thm:hurwitz}. The case of Teichm\"uller curves follows from the fact that $\SL_2(\bbR)$-orbits are never complete (see 
e.g. \cite[Proposition 3.2]{WrightLectures}), as the $\SL_2(\bbR)$-action can make a saddle connection arbitrarily short so that the underlying surface breaks in the end. Since any complete algebraic curve minus finitely many points is affine, we conclude that Teichm\"uller curves are affine. 

For the case of Hurwitz spaces, the idea of the proof is motivated by \cite{Fontanari}, which showed that the top stratum of Arbarello's Weierstrass flag of $\MM_g$ is affine (though \cite{ArbarelloMondello} showed that the rest strata of the Weierstrass flag are almost never affine). We set up some notation first. Let $S$ be a subgroup of the symmetric group $S_n$ which acts on $\MM_{1,n}$ by permuting the $n$ markings. In particular, $\MM_{1,n} / S_n$ is the moduli space of smooth genus one curves with $n$ unordered marked points. Let $\HH_d(\Pi)$ be the Hurwitz space of degree $d$ connected covers of genus one curves with a given ramification profile $\Pi$, that is, the number of branch points and the ramification type over each branch point are fixed. 
There is a finite morphism $f: \HH_d(\Pi) \to \MM_{1,n}/S$ for some subgroup $S$ of $S_n$ which permutes branch points of the same ramification type, and the degree of $f$ is the corresponding Hurwitz number. It remains to prove the following. 

\begin{theorem}
The Hurwitz space $\HH_d(\Pi)$ is affine.  
\end{theorem}

\begin{proof}
The composition $\HH_d(\Pi) \to \MM_{1,n}/S \to \MM_{1,n}/S_n$ remains to be a finite morphism, hence it suffices to show that $\MM_{1,n}/S_n$ is affine (see e.g. \cite[Chapter II, Corollary 1.5]{Hartshorne}). By \cite[Theorem 5.1]{ChenCoskun} the cone of effective divisors of the Deligne-Mumford compactification $\BM_{1,n}/S_n$ is spanned by the boundary divisors, hence the class of an ample divisor on $\BM_{1,n}/S_n$ must be a positive linear combination of the boundary divisor classes, as any ample divisor class lies in the interior of the effective cone. Therefore, $\MM_{1,n}/S_n$ is affine as the complement of an ample divisor in $\BM_{1,n}/S_n$. 
\end{proof}

\begin{remark}
If $\HH_d(\Pi)$ is reducible, the above proof applies to each connected component of $\HH_d(\Pi)$, hence every connected component of $\HH_d(\Pi)$ is affine. 
\end{remark}

\begin{remark}
\label{rem:example}
Except Teichm\"uller curves and Hurwitz spaces of torus coverings, $\GL_2^+(\bbR)$-orbit closures (other than the strata) are quite rare to find. Recently a totally geodesic Teichm\"uller surface was discovered in \cite{MMW}, which arises from the flex locus $F\subset \MM_{1,3}/S_3$. Since $\MM_{1,3}/S_3$ is affine, so is the closed subset $F$. 
\end{remark}

\section{Affine strata}
\label{sec:examples}

It is known as folklore to experts that the hyperelliptic strata and a number of low genus strata of holomorphic canonical divisors are affine (see \cite[Section 3.8]{Mondello}). For the sake of completeness, in this section we provide some detailed proofs. In addition, we reduce the affinity of strata of $k$-canonical divisors to the case of canonical divisors, see Proposition~\ref{prop:k-affine}. 

\subsection{The hyperelliptic strata}

The strata $\PP(2g-2)^{\hyp}$ and $\PP(g-1,g-1)^{\hyp}$ parameterize hyperelliptic curves with a marked Weierstrass point and a pair of distinct hyperelliptic conjugate points, respectively. For $a, b > 0$, let $\MM_{0, a;b}$ be the moduli space of smooth rational curves with $a$ ordered marked points and $b$ unordered marked points. Identifying a genus $g$ hyperelliptic curve with 
a smooth rational curve marked at the $2g+2$ branch points, it follows that $\PP(2g-2)^{\hyp}\cong \MM_{0, 1;2g+1}$ and $\PP(g-1,g-1)^{\hyp}\cong \MM_{0, 1;2g+2}$.
Let $\MM_{0, n}/S_n$ be the moduli space of smooth rational curves with $n$ unordered marked points, which is isomorphic to $\bbP^{n-3}$ minus a union of hypersurfaces, hence 
$\MM_{0, n}/S_n$ is affine. There is a finite morphism $\MM_{0, a;b} \to \MM_{0, a+b}/S_{a+b}$ which forgets the orders of the markings, hence $\MM_{0, a;b}$ is affine. In particular, $\PP(2g-2)^{\hyp}$ and $\PP(g-1,g-1)^{\hyp}$ are affine. 

\subsection{The stratum $\PP(4)^{\odd}$}

The first non-hyperelliptic stratum is $\PP(4)^{\odd}$ in genus three, which parameterizes non-hyperelliptic genus three curves $C$ with a hyperflex $p$, i.e., $4p\sim K_C$. Since $C$ is non-hyperelliptic, its canonical embedding is a smooth plane quartic. Then there exists a line $L$ tangent to $C$ at $p$ with contact order four. Projecting $C$ through $p$ to a general line induces a triple cover of $C$ to $\bbP^1$. By Riemann-Hurwitz, there exists (at least) two other lines $L_1$ and $L_2$ through $p$, which are tangent to $C$ at $p_1$ and $p_2$, respectively. Consider the resulting point-line configuration in $\bbP^2$: $(p \in L, L_1, L_2; \ p_1 \in L_1; \ p_2 \in L_2)$ 
where $p, p_1, p_2$ are distinct points and $L, L_1, L_2$ are distinct lines. It is easy to check that $\GL(3)$ acts on the set of such configurations transitively. 

Let $[X, Y, Z]$ be the homogeneous coordinates of $\bbP^2$. Without loss of generality, we may assume that $L$ is defined by $Y = 0$, $L_1$ is defined by $X = 0$, $L_2$ is defined by $X - Y = 0$, $p = [0, 0, 1]$, $p_1 = [0,1,0]$ and $p_2 = [1,1,0]$. The subgroup $G$ of $\GL(3)$ preserving this configuration consists of elements of the form $\diag (\mu, \mu, \lambda)$ where $\mu, \lambda \in \bbC^{*}$. Let $F = \sum^{i+j+k=4}_{i,j,k\geq 0} a_{ijk} X^i Y^j Z^k$ be the defining equation of $C$. The tangency conditions imposed by the point-line configuration imply that 
$$ a_{004} = a_{103} = a_{202} = a_{301} = 0, $$
$$ a_{040} = a_{031} = 0, $$
$$ a_{400} + a_{310} + a_{220} + a_{130} = 0, $$
$$ a_{301} + a_{211} + a_{121} + a_{031} = 0 $$
which impose eight independent conditions to the space $\bbP^{14}$ of all plane quartics. We can take $a_{400}, a_{310}, a_{220}, a_{211}, a_{121}, a_{022}, a_{013}$ as free parameters and $G$ acts on them by taking $a_{ijk}$ to $\mu^{i+j}\lambda^k a_{ijk}$. Note that if $a_{400} = 0$, then $F$ is divisible by $Y$, and $C$ would be reducible. Since affinity is preserved under finite morphisms and $G$ takes $a_{400}$ to $\mu^4 a_{400}$, after a degree four base change we can assume that $a_{400} = 1$ and $\mu = 1$. Then the space of remaining parameters with the $G$-action can be identified with $\bbC^2 \times (\bbC^4\setminus \{ 0\} / G)$ minus a (possibly reducible) hypersurface which parameterizes singular quartics, where the first term $\bbC^2$ parameterizes $a_{310}$ and $a_{220}$, the second term $\bbC^4$ parameterizes $(a_{211}, a_{121}, a_{022}, a_{013})$ (not all zero for otherwise $C$ would be reducible) and $G$ acts on it by taking $a_{ijk}$ to $\lambda^k a_{ijk}$ for $\lambda \in \bbC^{*}$. Since $(\bbC^4\setminus \{ 0\} / G)$ is a weighted projective space, the complement of a hypersurface is affine. Therefore, up to the choice of the two specified tangent lines $L_1$ and $L_2$, $\PP(4)^{\odd}$ can be realized as the image of a finite morphism from an affine variety, hence it is affine.  

\subsection{The stratum $\PP(3,1)$}

This stratum parameterizes non-hyperelliptic genus three curves $C$ with two distinct points $p$ and $q$ such that $3p + q \sim K_C$. In terms of the canonical embedding of $C$ as a smooth plane quartic, there exists a line $L$ tangent to $C$ at $p$ with contact order three, and $q$ is determined by the other intersection point of $L$ with $C$. Take the same point-line configuration as in the previous case to impose tangency conditions to $C$. Then we have 
$$ a_{004} = a_{103} = a_{202} = 0, $$
$$ a_{040} = a_{031} = 0, $$
$$ a_{400} + a_{310} + a_{220} + a_{130} = 0, $$
$$ a_{301} + a_{211} + a_{121} + a_{031} = 0 $$
which impose seven independent conditions to the space $\bbP^{14}$ of all plane quartics. We can take $a_{130}, a_{220}, a_{310}, a_{301}, a_{211}, a_{121}, a_{022}, a_{013}$ as free parameters and $G$ acts on them by taking $a_{ijk}$ to $\mu^{i+j}\lambda^k a_{ijk}$. Note that if $a_{130} = 0$, then $C$ would be singular at $p_1 = [0,1,0]$. Since affinity is preserved under finite morphisms and $G$ takes $a_{130}$ to $\mu^4 a_{130}$, after a degree four base change we can assume that $a_{130} = 1$ and $\mu = 1$. Then the space of remaining parameters with the $G$-action can be identified with $\bbC^2 \times (\bbC^5\setminus \{ 0\} / G)$ minus a union of hypersurfaces which parameterize singular quartics and quartics that have contact order four to $L$ at $p$, where the first term $\bbC^2$ parameterizes $a_{220}$ and $a_{310}$,  the second term $\bbC^5$ parameterizes $a_{301}, a_{211}, a_{121}, a_{022}, a_{013}$ (not all zero for otherwise $C$ would be reducible) and $G$ acts on it by taking $a_{ijk}$ to $\lambda^k a_{ijk}$ for $\lambda \in \bbC^{*}$. Since $(\bbC^5\setminus \{ 0\} / G)$ is a weighted projective space, the complement of hypersurfaces is affine. Therefore, up to the choice of the two specified tangent lines $L_1$ and $L_2$, $\PP(3,1)$ can be realized as the image of a finite morphism from an affine variety, hence it is affine. 

\subsection{The stratum $\PP(2,2)^{\odd}$}

This stratum parameterizes non-hyperelliptic genus three curves $C$ with two distinct points $p_1$ and $p_2$ such that $2p_1 + 2p_2 \sim K_C$. In terms of the canonical embedding of $C$ as a smooth plane quartic, there exists a line $L$ tangent to $C$ at $p_1$ and $p_2$. Since every smooth plane quartic has $28$ distinct bitangent lines, let $M$ be another bitangent line that intersects $C$ at $q_1$ and $q_2$. 
Denote by $r$ the intersection of $L$ with $M$. Note that $r$ is distinct from $p_i$ and $q_j$. Consider the resulting point-line configuration in $\bbP^2$: 
$(p_1, p_2 \in L; \ q_1, q_2 \in M;\ r = L\cap M )$,  
where $p_1, p_2, q_1, q_2, r$ are distinct points and $L, M$ are distinct lines. It is easy to check that $\GL(3)$ acts on the set of such configurations transitively. Without loss of generality, we may assume that $L$ is defined by $Y = 0$, $M$ is defined by $Z = 0$, $r = [1, 0, 0]$, $p_1 = [0,0,1]$, $p_2 = [1,0,1]$, $q_1 = [0,1,0]$ and $q_2 = [1,1,0]$. 
 The subgroup $G$ of $\GL(3)$ preserving this configuration consists of scalar matrices only. Following the previous notation, a routine calculation shows that these bitangency requirements impose eight linearly independent conditions to the coefficients $a_{ijk}$, hence in the total space of plane quartics these conditions cut out a subspace $\bbP^6$ which contains a hypersurface of singular quartics. Therefore, up to the choice of the two specified bitangent lines, $\PP(2,2)^{\odd}$ can be realized as the image of a finite morphism from an affine variety, hence it is affine. 

\begin{remark}
For the two remaining strata $\PP(2,1,1)$ and $\PP(1,1,1,1)$ in genus three, they both contain hyperelliptic curves whose canonical maps are double covers of plane conics, hence the plane quartic model does not directly apply to these cases. Nevertheless, one can use the above method to show that the complements of the hyperelliptic locus
in these two strata are affine.   
\end{remark}

\subsection{The stratum $\PP(6)^{\even}$}

This stratum parameterizes non-hyperelliptic genus four curves $C$ with a point $p$ such that $6p \sim K_C$ and $h^0(C, 3p) = 2$. In terms of the canonical embedding of $C$ 
in $\bbP^3$, there exists a unique quadric cone $Q$ containing $C$ such that a ruling $L$ of $Q$ cuts out $3p$ with $C$. Let $v$ be the vertex of $Q$. There exists (at least) two other rulings $L_1$ and $L_2$ through $v$ such that $L_i$ is tangent to $C$ at $p_i$ for $i = 1, 2$. Consider the resulting configuration: $(p \in L, p_1 \in L_1, p_2 \in L_2; \ L, L_1, L_2 \subset Q)$, where $p, p_1, p_2$ are distinct points from $v$ and $L, L_1, L_2$ are distinct lines. It is easy to check that $\GL(4)$ acts on the set of such configurations transitively. Let $[X, Y, Z, W]$ be the homogeneous coordinates of $\bbP^3$. Without loss of generality, we may assume that $Q$ is defined by $XY - Z^2 = 0$, $L$ is defined by 
$X = Y = Z$, $L_1$ is defined by $Y= Z = 0$, $L_2$ is defined by $X = Z = 0$, $p = [1,1,1,0]$, $p_1 = [1, 0, 0, 0]$ and $p_2 = [0,1,0,0]$. The subgroup $G$ of $\GL(4)$ preserving this configuration consists of elements of the form $\diag( \lambda,  \lambda,  \lambda, \mu)$ 
where $\lambda, \mu \in \bbC^{*}$. Let $F = \sum_{i+j+k=3} a_{ijk} X^i Y^j W^k + Z \sum_{i+j+k=2} b_{ijk}X^i Y^j W^k$ be the cubic equation that cuts out $C$ in $Q$.
A routine calculation shows that these tangency requirements impose seven linearly independent conditions to the coefficients $a_{ijk}$ and $b_{ijk}$: 
$$ a_{300} = a_{201} = a_{030} = a_{021} = 0, $$
$$ a_{210} + a_{120} + b_{200} + b_{110} + b_{020} = 0, $$
$$a_{111} + b_{101} + b_{011} = 0, $$
$$a_{102} + a_{012} + b_{002} = 0. $$ 
We can take $a_{003}, a_{120}, b_{200}, b_{110}, b_{020}, b_{101}, b_{011}, a_{012}, b_{002}$ as free parameters and $G$ acts on them by taking $a_{ijk}$ and $b_{ijk}$ to $\lambda^{i+j} \mu^k a_{ijk}$ and $\lambda^{i+j+1} \mu^k b_{ijk}$, respectively. Since $C$ does not contain $v$, it implies that $a_{003} \neq 0$. After a degree three base change we can assume 
that $a_{003}= 1$ and $\mu = 1$. Then the space of remaining parameters with the $G$-action can be identified with a weighted projective space $\bbC^8\setminus \{ 0\} / G$, where it contains a hypersurface parameterizing singular curves and $G$ takes $a_{ijk}$ and $b_{ijk}$ to $\lambda^{i+j} a_{ijk}$ and $\lambda^{i+j+1} b_{ijk}$, respectively. Therefore, up to the choice of the specified rulings, $\PP(6)^{\even}$ can be realized as the image of a finite morphism from an affine variety, hence it is affine. 

\begin{remark}
For the other strata in genus four, they may contain both hyperelliptic and non-hyperelliptic curves, or both Gieseker-Petri general and Gieseker-Petri special curves, hence one cannot use a single quadric surface to treat them. 
\end{remark}

\subsection{Strata of $k$-canonical divisors}

One can similarly ask if the strata of (possibly meromorphic) $k$-canonical divisors $\PP^k(\mu)$ are affine. Via the canonical cyclic $k$-cover discussed at the end of Section~\ref{sec:complete}, $\PP^k(\mu)$ lifts into a stratum of canonical divisors $\PP(\wt{\mu})$ as a closed subset. We thus conclude the following result. 

\begin{proposition}
\label{prop:k-affine}
If the strata of canonical divisors are affine, then the strata of $k$-canonical divisors are affine for all $k$. 
\end{proposition}


\end{document}